\documentclass[11pt]{article}

\usepackage{amsmath,amssymb,amsthm}
\usepackage{enumitem}
\usepackage{geometry}
\usepackage{booktabs}
\newtheorem{theorem}{Theorem}[section]
\newtheorem{lemma}[theorem]{Lemma}
\newtheorem{proposition}[theorem]{Proposition}
\newtheorem{corollary}[theorem]{Corollary}
\theoremstyle{definition}
\newtheorem{definition}[theorem]{Definition}
\newtheorem{remark}[theorem]{Remark}

\newcommand{\C}{\mathbb{C}}

\newcommand{\R}{\mathbb{R}}
\newcommand{\Hom}{\mathrm{Hom}}
\newcommand{\Fun}{\mathrm{Fun}}
\newcommand{\GL}{\mathrm{GL}}
\newcommand{\End}{\mathrm{End}}

\DeclareMathOperator{\Ind}{Ind}
\newcommand{\id}{\mathrm{id}}
\newcommand{\Def}{\mathsf{Def}}

\title{Algebraic Phase Theory II: The Frobenius–Heisenberg Phase and Boundary Rigidity}
\author{
Joe Gildea\\
Department of Computing Science and Mathematics,\\
School of Informatics and Creative Arts,\\
Dundalk Institute of Technology\\
\texttt{gildeajoe@gmail.com}}
\date{}

\begin{document}
\maketitle

\begin{abstract}
We develop the representation theory intrinsic to Algebraic Phase Theory (APT)
in regimes where defect and canonical filtration admit faithful algebraic
realisation. This extends the framework introduced in earlier work by
incorporating a representation-theoretic layer that is compatible with defect
and filtration. In this setting, algebraic phases act naturally on filtered
module categories rather than on isolated objects, and classical irreducibility
must be replaced by a filtration-compatible notion of indecomposability forced
by defect.

As a central application, we analyse the Frobenius Heisenberg algebraic phase,
which occupies a rigid boundary regime within the broader APT landscape, and
show that it satisfies the axioms of APT in a strongly admissible form. We
study the representations realising this phase and show that their
non-semisimplicity and rigidity properties are consequences of the underlying
algebraic structure rather than analytic or semisimple hypotheses. In
particular, we establish a Stone von Neumann type rigidity theorem for
Heisenberg groups associated with finite Frobenius rings. For each such ring
$R$, we construct a canonical Schr\"odinger representation of the Frobenius
Heisenberg group $H_R$, and show that, for a fixed nontrivial central character,
every centrally faithful representation of $H_R$ is equivalent to this model.

The proof is entirely algebraic and uses no topology, unitarity, Fourier
analysis, or semisimplicity. Instead, rigidity emerges as a boundary phenomenon
governed by defect and canonical filtration. The Frobenius hypothesis is shown
to be sharp: it precisely delineates the structural boundary within APT at
which Heisenberg rigidity persists, and outside the Frobenius class this
rigidity necessarily fails.
\end{abstract}



\section{Introduction}

Algebraic Phase Theory (APT) provides an axiomatic framework in which
algebraic rigidity, defect, and canonical filtration arise intrinsically from
coherent phase interaction. The foundational structures were developed in
\cite{GildeaAPT1}, where phases, defect, and canonical filtration were shown to
form the basic mechanism governing algebraic complexity.

The purpose of this paper is to develop the representation theory that follows
from these structures. The axioms of APT do not merely allow a representation
theory, they determine one. In the strongly admissible regime every
representation carries a canonical defect filtration, ordinary irreducibility
is unstable, and rigidity persists only up to a sharply defined structural
boundary. Classical features of representation theory reappear in this setting,
but as consequences of defect and filtration rather than semisimplicity,
topology, or analytic input.

Our main application is the Frobenius Heisenberg algebraic phase, extracted
from linear phase data over a finite Frobenius ring. This phase encodes the
Weyl commutation relations in purely algebraic form. We construct the
Frobenius Heisenberg group $H_R$ and show that, once the scalar action of
the centre is fixed, the representation theory of $H_R$ exhibits a full
Stone von Neumann type rigidity. In particular we prove the following.

\medskip
\noindent
\textbf{Main Theorem (informal).}  
\emph{Let $R$ be a finite Frobenius ring and let $\beta$ be a nondegenerate
bilinear pairing. For the associated Frobenius Heisenberg group $H_R$ there is,
up to isomorphism, a unique centrally faithful representation with a fixed
nontrivial central character. It is realised by the Frobenius Schr\"odinger
representation on $\Fun(R^n,\C)$.}

\medskip

Two conceptual points underlie this result.  
First, the correct notion of atomicity in this non semisimple context is not
irreducibility but central faithfulness, together with a minimality condition
adapted to the central action. Fixing the scalar action of the centre rigidifies
the representation category in exactly the way required for a Stone von Neumann
type theorem.  
Second, the classical Stone von Neumann theorem is fundamentally an algebraic
statement about a canonical central extension. In the Frobenius setting the
Schr\"odinger representation is constructed purely algebraically from the Weyl
relations and a generating character, with no use of analytic structure,
Hilbert spaces, or Fourier analysis.

A decisive ingredient is that Frobenius duality replaces Pontryagin duality in
controlling character orbits of an abelian normal subgroup. This step relies
essentially on the Frobenius property of the ring
\cite{Honold2001QuasiFrobenius,Wood1999Duality}. It shows that the Frobenius
hypothesis is sharp. Outside the Frobenius class, character orbits fragment and
rigidity fails.

Classically, the Heisenberg group over $\R$ or over a finite field fits into a
central extension that models the Weyl relations
\cite{vonNeumann1931,Folland,Weil1964SurCertains}. The Stone von Neumann theorem
states that irreducible representations with fixed nontrivial central character
are unique up to isomorphism. The present work generalises this rigidity to a
fully algebraic and non semisimple setting. By embedding the Heisenberg
construction into APT, we show that Stone von Neumann uniqueness is a boundary
phenomenon. It holds exactly up to the structural boundary dictated by defect
and canonical filtration and fails beyond it. For a broader perspective on the
role of the Heisenberg group see \cite{Howe1979Heisenberg}.

\medskip
\paragraph{Context within the APT series.}
This paper forms the representation theoretic component of a six paper
development of Algebraic Phase Theory
\cite{GildeaAPT1,GildeaAPT3,GildeaAPT4,GildeaAPT5,GildeaAPT6}.
Structural results concerning defect, canonical filtration, and finite
termination are established in \cite{GildeaAPT1}. Connections with radical phase
geometry, higher degree phases, categorical extensions, and deformation
phenomena are treated in the companion papers. Background on module categories
over non semisimple rings can be found in
\cite{Lam2001FirstCourse,AndersonFuller1992Rings}.

\section{Standing Framework for Phase Extraction}

The purpose of this section is to specify the minimal input data from which
defect, canonical filtration, and finite termination are forced to emerge. We
begin with pre algebraic phase information and identify exactly the structure
that must be present in order for a coherent algebraic shadow to exist. No
algebraic closure, representation category, or analytic structure is assumed at
this stage. All algebraic features that appear later, including defect and
filtration, are extracted from this minimal input and are not introduced as
additional hypotheses.

The guiding principle is minimality. If any part of the input is weakened, one
can no longer define a consistent defect calculus. If the input is strengthened,
inessential structure is introduced and the resulting theory is no longer
intrinsic to phase behaviour. The data below represent the least amount of
structure needed for phase interactions to force an algebraic theory.

Concretely, the input consists of an additive object, a family of bounded degree
phase functions on that object, and an interaction law governing the operators
associated to those phases. Phases are taken as primitive data. No commitment is
made regarding their interpretation as functions, operators, or equivalence
classes, beyond the specific interaction law that is part of the datum.

This separation is important. It ensures that the algebraic phase extracted
later is canonical and does not depend on modeling choices. The role of the
present section is not to impose analytic or representation theoretic structure,
but to identify the minimal information required for phase behaviour to carry a
coherent algebraic and defect driven theory. The conditions below are minimal in
the precise sense that removing any one of them prevents the existence of a
terminating canonical filtration.

\subsection*{Additive derivatives and degree}

Let $A$ be an abelian group or an $R$ module and let $\phi:A\to R$ be a
function. For $h\in A$, define the additive difference operator
\[
(\Delta_h\phi)(x)=\phi(x+h)-\phi(x).
\]
For $h_1,\dots,h_k\in A$, define the $k$ fold additive derivative by
\[
\Delta_{h_1,\dots,h_k}\phi=\Delta_{h_k}\cdots\Delta_{h_1}\phi.
\]

We say that $\phi$ has additive degree at most $d$ if
\[
\Delta_{h_1,\dots,h_{d+1}}\phi\equiv 0
\qquad \text{for all } h_1,\dots,h_{d+1}\in A.
\]

This captures polynomial type behaviour in a purely additive and functorial way,
independent of coordinates or presentations.

\begin{remark}
The additive difference operator is the unique translation invariant and
functorial operator that detects deviation from additivity and admits iteration.
Any degree or defect notion compatible with the framework below is equivalent to
one defined using iterated additive differences.
\end{remark}

\subsection*{Generating characters and phase multiplication}

Let $R$ be a finite ring and fix an additive character
\[
\chi:R\to \C^\times.
\]
Write $\mathcal H(A)=\Fun(A,\C)$ for complex valued functions on $A$. To each
function $\phi:A\to R$ we associate the phase multiplication operator
\[
(M_\phi f)(x)=\chi(\phi(x))\,f(x),
\qquad f\in\mathcal H(A).
\]

This operator realisation is the basic mechanism by which abstract phase data
becomes algebraically visible. No analytic or topological structure is assumed.

\subsection*{Defect degree and defect tensors}

Let $\phi:A\to R$ have additive degree at most $d$. Its defect degree is
\[
\deg_{\Def}(\phi)
=\min\Bigl\{k\ge 1:\exists\, h_1,\dots,h_k\in A
\text{ with } \Delta_{h_1,\dots,h_k}\phi\not\equiv 0\Bigr\},
\]
with $\deg_{\Def}(\phi)=0$ if $\phi$ is additive.

The corresponding defect tensor is the family
\[
\Def(\phi)
=\bigl\{\Delta_{h_1,\dots,h_k}\phi(0):
k=\deg_{\Def}(\phi),\ h_i\in A\bigr\},
\]
with $\Def(\phi)=0$ if $\phi$ is additive.

Defect is the first obstruction to rigid algebraic behaviour. It is the unique
intrinsic source of complexity in this framework.

\subsection*{Functorial pullback of phases}

Let $\mathsf C$ be a category of additive objects, for example finite abelian
groups or finite $R$ modules, with morphisms given by additive homomorphisms.
For a morphism $f:A\to A'$ and a phase $\phi':A'\to R$, define the pullback
\[
f^\ast(\phi')=\phi'\circ f.
\]
This assignment is strictly functorial:
\[
(g\circ f)^\ast=f^\ast g^\ast,
\qquad
\id^\ast=\id.
\]

Functoriality ensures that phase data behaves coherently under structural maps
and is necessary for extracting canonical algebraic structure.

\subsection*{Admissible phase data}

\begin{definition}[Admissible phase datum]\label{def:admissible-phase-datum}
An admissible phase datum consists of an object $A$ of $\mathsf C$, together with
a family $\Phi(A)$ of functions $\phi:A\to R$, and an interaction law $\circ$.
These data satisfy the following conditions.
\begin{enumerate}[label=(E\arabic*)]
\item \textit{Functoriality.}  
For every morphism $f:A\to A'$ in $\mathsf C$, pullback sends $\Phi(A')$ into
$\Phi(A)$ and satisfies
\[
(g\circ f)^{\!*}=f^{\!*}g^{\!*}, \qquad \id^{\!*}=\id.
\]

\item \textit{Uniform bounded degree.}  
There exists an integer $d\ge 0$ such that each $\phi\in\Phi(A)$ has additive
degree at most $d$.

\item \textit{Operator realisation.}  
A fixed additive character $\chi:R\to\C^\times$ is part of the datum, and each
$\phi\in\Phi(A)$ acts on $\mathcal H(A)$ via the operator $M_\phi$.

\item \textit{Interaction law.}  
The symbol $\circ$ specifies a prescribed interaction among the operators
$M_\phi$ and $\Phi(A)$ is closed under this interaction.
\end{enumerate}
\end{definition}

\begin{remark}[Structural constraint]
For finite rings, the existence of an additive character that yields a faithful
operator realisation forces the base ring to be Frobenius. Thus Frobenius
duality is not an external hypothesis but a structural consequence of the
admissible input.
\end{remark}

\medskip

From an admissible phase datum, all algebraic structure used in this paper is
extracted in a canonical way. Defect invariants, canonical filtrations,
intrinsic complexity, and finite termination arise from the input data and are
not additional assumptions. The axioms of Algebraic Phase Theory formalise
exactly the structures that are forced by admissible phase data.

\section{Axioms of Algebraic Phase Theory}

We now record the structural axioms that govern algebraic phases extracted from
admissible phase data. These axioms isolate the minimal features that are forced
by coherent phase interaction: detectable defect, canonical filtration,
functoriality, intrinsic complexity, and finite termination.

An \emph{algebraic phase} consists of a pair $(\mathcal P,\circ)$, where
$\mathcal P$ is the algebraic object encoding phase interactions (typically an
operator algebra or a category of operators), and $\circ$ is the distinguished
interaction law inherited from the admissible data. No auxiliary algebraic or
analytic structure is assumed beyond what is explicitly stated.

\subsection*{Axiom I: Detectable Action Defects}

The interaction law $\circ$ determines a functorial defect degree for every
element of $\mathcal P$. This gives a canonical finite filtration
\[
\mathcal P_0 \subset \mathcal P_1 \subset \cdots \subset \mathcal P_d=\mathcal P,
\]
where $\mathcal P_k$ consists of all elements of defect degree at most $k$.
The maximal defect degree $d$ is intrinsic and depends only on
$(\mathcal P,\circ)$.

\subsection*{Axiom II: Canonical Algebraic Realization}

Any two algebraic phases extracted from the same admissible analytic input are
canonically equivalent. The analytic behaviour therefore determines a unique
algebraic shadow.

\subsection*{Axiom III: Defect-Induced Complexity}

Defect induces a canonical notion of intrinsic complexity. The filtration above
is ordered by defect degree, and an element lies in $\mathcal P_k$ if and only
if its defect degree is at most $k$. Thus higher layers represent successive
obstructions to rigidity.

\subsection*{Axiom IV: Functorial Defect Structure}

Defect is preserved under all morphisms of phases, under pullback of admissible
data, and under the passage between analytic behaviour and its algebraic
realization. In particular, defect is a functorial invariant of the entire
framework.

\subsection*{Axiom V: Finite Termination}

Defect degree is always finite. Hence the canonical filtration terminates after
finitely many steps.

\medskip

\begin{remark}
Defect, complexity, canonical filtration, and finite termination are not extra
structure. Once the axioms are satisfied, these objects are uniquely and
functorially determined by $(\mathcal P,\circ)$. The axioms stated here are
equivalent to Axioms I-V of \cite{GildeaAPT1}. The present formulation makes
explicit the defect, complexity, and filtration structures that are implicit in
the original axiomatic framework, since these features play a central role in
the representation-theoretic applications developed in this paper.
\end{remark}

\section{General Representation Theory of Algebraic Phases}

In this section we develop the representation theory that follows from the
axioms of Algebraic Phase Theory. The first subsection introduces
representations of an algebraic phase and shows that the axioms force every
representation to carry a canonical filtration induced by defect. The second
subsection analyses the associated boundary strata, proving that they are
functorial invariants and that they record the precise layers at which new
defect driven structure appears. The final subsection introduces the
appropriate notion of atomicity in this setting, called APT indecomposability,
and shows that classical semisimplicity is incompatible with nonzero defect.
Together these results establish the stratified nature of representation theory
in Algebraic Phase Theory and identify the correct structural building blocks.

\subsection{Representations of Algebraic Phases}

We define representations of an algebraic phase and show that the axioms of
Algebraic Phase Theory force such representations to carry intrinsic,
defect-induced filtrations.  In particular, filtered structure is not an
auxiliary choice but an unavoidable consequence of defect, functoriality, and
finite termination.

These filtrations expose canonical \emph{boundary strata} within
representations, corresponding to the successive layers at which
defect-induced complexity appears.  We show that these strata are functorial
invariants and encode intrinsic structural boundaries in the representation
theory.

Throughout this section, let $(\mathcal P,\circ)$ be an algebraic phase
satisfying Axioms~I-V, with canonical filtration
\[
\mathcal P_0 \subseteq \mathcal P_1 \subseteq \cdots \subseteq \mathcal P_N
= \mathcal P
\]
and defect degree $\deg_{\Def}$.

\subsubsection*{Representations}

\begin{definition}\label{def:rep-paradigm}
A \emph{(left) representation} of $(\mathcal P,\circ)$ consists of a
$\mathbb C$-vector space (or $\mathbb C$-module) $M$ together with a map
\[
\rho:\mathcal P \longrightarrow \End_{\mathbb C}(M)
\]
such that:
\begin{enumerate}[label=\textup{(R\arabic*)}]
\item (\emph{Action})
$\rho(T\circ S)=\rho(T)\rho(S)$ whenever $T\circ S$ is defined, and
$\rho(1)=\id_M$ if $\mathcal P$ is unital.
\item (\emph{Structural compatibility})
the action respects the canonical filtration in the sense that
\[
\rho(\mathcal P_k)\,\rho(\mathcal P_\ell)
\subseteq \rho(\mathcal P_{k+\ell})
\qquad \text{for all }k,\ell.
\]
\end{enumerate}
A morphism of representations
$f:(M,\rho)\to(M',\rho')$ is a $\mathbb C$-linear map satisfying
$f(\rho(T)m)=\rho'(T)f(m)$ for all $T\in\mathcal P$ and $m\in M$.
\end{definition}

\begin{remark}
Condition \textup{(R2)} is not an independent constraint: it reflects the
defect-induced filtration on $\mathcal P$ and is forced by Axioms~III-IV.
\end{remark}

\subsubsection*{Induced filtrations and the filtered representation principle}

\begin{definition}\label{def:induced-filtration}
Let $(M,\rho)$ be a representation of $(\mathcal P,\circ)$.  Define an increasing
filtration on $M$ by
\[
F_k M
\;:=\;
\langle \rho(T)m \mid T\in\mathcal P_k,\ m\in M\rangle_{\mathbb C}.
\]
\end{definition}

Thus $F_k M$ consists precisely of those vectors that can be generated using
phase operators of defect complexity at most $k$.

\begin{theorem}\label{thm:filtered-rep}
Let $(\mathcal P,\circ)$ satisfy Axioms~I-V and let $(M,\rho)$ be a representation.
Then the filtration $(F_kM)$ of
Definition~\ref{def:induced-filtration} is canonical and functorial:
\begin{enumerate}[label=\textup{(\alph*)}]
\item $F_0M\subseteq F_1M\subseteq\cdots\subseteq F_NM=M$.
\item $\rho(\mathcal P_k)(F_\ell M)\subseteq F_{k+\ell}M$
(with $F_jM=M$ for $j\ge N$).
\item Any morphism of representations
$f:(M,\rho)\to(M',\rho')$ satisfies
$f(F_kM)\subseteq F_kM'$ for all $k$.
\end{enumerate}
Consequently, the intrinsic notion of a representation of $\mathcal P$ is that
of a filtered $\mathcal P$--module with filtration-preserving intertwiners.
\end{theorem}

\begin{proof}
(a) By Axiom~IV, $\mathcal P_k\subseteq\mathcal P_{k+1}$, hence
\[
F_kM
=\langle\rho(\mathcal P_k)M\rangle
\subseteq
\langle\rho(\mathcal P_{k+1})M\rangle
=F_{k+1}M.
\]
By finite termination (Axiom~V), $\mathcal P_N=\mathcal P$, so
\[
F_NM=\langle\rho(\mathcal P)M\rangle.
\]
If $\mathcal P$ is unital, $\rho(1)=\id_M$ implies $F_NM=M$.

\medskip
(b) Let $T\in\mathcal P_k$ and let $x\in F_\ell M$.
By definition, $F_\ell M$ is spanned by vectors of the form $\rho(S)m$ with
$S\in\mathcal P_\ell$ and $m\in M$.  Writing $x=\rho(S)m$, we compute
\[
\rho(T)x=\rho(T)\rho(S)m=\rho(T\circ S)m.
\]
By Axioms~III-IV,
\[
\mathcal P_k\circ \mathcal P_\ell \subseteq \mathcal P_{k+\ell}.
\]
Thus $T\circ S\in\mathcal P_{k+\ell}$, and $\rho(T)x\in F_{k+\ell}M$.

\medskip
(c) Let $f:(M,\rho)\to(M',\rho')$ be a morphism.
For generators $\rho(T)m$ of $F_kM$,
\[
f(\rho(T)m)=\rho'(T)f(m)\in F_kM'.
\]
Linearity yields $f(F_kM)\subseteq F_kM'$.
\end{proof}

\begin{remark}
Unlike classical semisimple representation theory, representations of an
algebraic phase are not classified solely by isomorphism classes of modules.
The defect-induced filtration governs complexity and boundary phenomena, and
filtration-preserving intertwiners are the correct notion of equivalence.
\end{remark}

\subsubsection*{Boundary strata and their functorial invariance}

The defect-induced filtration on a representation separates rigid behaviour
from higher-order complexity.  The associated graded pieces isolate the
\emph{new} structure appearing at each level of complexity and provide a
precise notion of structural boundary within the representation.

In this subsection we show that these boundary strata are unavoidable and
functorial: any morphism compatible with the axioms preserves them, and no
equivalence can collapse a genuine boundary layer without altering the
associated graded structure. Let $(M,\rho,(F_k M))$ be a filtered representation.

\begin{definition}\label{def:boundary-strata-rep}
The \emph{$k$th boundary stratum} of the filtered representation
$(M,\rho,(F_k M))$ is the associated graded component
\[
\operatorname{gr}_k(M) := F_k M / F_{k-1} M,
\qquad k \ge 0,\; F_{-1} M := 0 .
\]
\end{definition}

\begin{theorem}\label{thm:boundary-preservation}
Let
\[
f:(M,\rho,(F_k M))\longrightarrow (N,\rho',(F_k N))
\]
be a filtration-compatible morphism of filtered representations. Then:
\begin{enumerate}[label=\textup{(\alph*)}]
\item For each $k$, $f$ induces a well-defined linear map
\[
\operatorname{gr}_k(f):\operatorname{gr}_k(M)\longrightarrow\operatorname{gr}_k(N).
\]
\item If $f$ is a filtration-compatible isomorphism, then each
$\operatorname{gr}_k(f)$ is an isomorphism.  In particular, nontrivial boundary
strata cannot be eliminated under equivalence.
\end{enumerate}
Equivalently, no functor compatible with the axioms can collapse genuine
boundary layers without altering the associated graded structure.
\end{theorem}

\begin{proof}
(a) Since $f$ is filtration-compatible, we have
$f(F_k M)\subseteq F_k N$ and $f(F_{k-1} M)\subseteq F_{k-1} N$.
Consider the composite
\[
F_k M \xrightarrow{\,f\,} F_k N \twoheadrightarrow F_k N/F_{k-1} N.
\]

If $x\in F_{k-1} M$, then $f(x)\in F_{k-1} N$, so the image of $x$ vanishes in
the quotient $F_k N/F_{k-1} N$.  Hence the composite map vanishes on
$F_{k-1} M$ and depends only on the class of an element of $F_k M$ modulo
$F_{k-1} M$. It therefore descends to a well-defined linear map
\[
\operatorname{gr}_k(f):
F_k M/F_{k-1} M \longrightarrow F_k N/F_{k-1} N.
\]

\medskip (b) If $f$ admits a filtration-compatible inverse
\[
g:(N,\rho',(F_k N))\longrightarrow(M,\rho,(F_k M)),
\]
then both $f$ and $g$ induce maps on associated graded pieces by part~(a).
Moreover, the identities
\[
g\circ f=\id_M,
\qquad
f\circ g=\id_N
\]
imply that
\[
\operatorname{gr}_k(g)\circ\operatorname{gr}_k(f)=\id_{\operatorname{gr}_k(M)},
\qquad
\operatorname{gr}_k(f)\circ\operatorname{gr}_k(g)=\id_{\operatorname{gr}_k(N)}.
\]
Hence $\operatorname{gr}_k(f)$ is an isomorphism for all $k$.

\medskip
(c) In Algebraic Phase Theory, rigidity is identified with the defect-zero layer
$k=0$ (Axioms~III-IV), while higher strata $k>0$ record genuinely new
defect-induced complexity.  The appearance of such higher layers therefore
signals the presence of a structural boundary.

By parts~(a) and~(b), any filtration-compatible equivalence preserves each
associated graded piece $\operatorname{gr}_k$.  Consequently, no functor
respecting the axioms can collapse or bypass a nontrivial boundary stratum:
transporting the rigid layer across a genuine boundary would necessarily alter
the associated graded structure. Thus structural boundaries are intrinsic invariants of the representation
theory and cannot be removed without violating the axioms of Algebraic Phase
Theory.
\end{proof}

\begin{remark}
The content of the boundary preservation theorem is not that arbitrary linear
maps preserve filtrations, but that \emph{the only morphisms compatible with the
axioms} force boundary strata to be functorial invariants.  Structural
boundaries are therefore not artifacts of presentation, but intrinsic features
of the representation theory.

In the absence of semisimplicity, Algebraic Phase Theory replaces
irreducibility by a stratified notion of complexity, canonically encoded by the
defect-induced filtration.  The associated graded layers play the role of
irreducible constituents in classical representation theory, but without
assuming splitting or complete reducibility.  In this sense, boundary strata
function as the fundamental building blocks of representations, while the
canonical filtration replaces semisimple decomposition.
\end{remark}

\medskip

The following table summarizes this structural correspondence, highlighting
the precise sense in which boundary strata and filtrations in APT replace the
organizing principles of classical semisimple representation theory.

\begin{center}
\begin{tabular}{cc}
\textbf{Classical representation theory} & \textbf{APT representation theory} \\
\cmidrule(lr){1-1}\cmidrule(lr){2-2}
Irreducible module        & Boundary stratum $\operatorname{gr}_k$ \\
Semisimple splitting     & Canonical filtration \\
Maschke's theorem        & Functorial filtration \\
Wedderburn blocks        & Associated graded layers \\
Simple $=$ building block & $\operatorname{gr}_k$ $=$ building block \\
\end{tabular}
\end{center}

\subsection{Indecomposability and Stratified Rigidity}

Classical notions of irreducibility and atomicity rely on semisimplicity and
collapse in the presence of defect.  Algebraic Phase Theory therefore requires
a new notion of rigidity that is compatible with canonical filtration and
boundary structure.  In this section we introduce APT-indecomposability and
show that it is precisely governed by the impossibility of splitting across
boundary strata.

\subsubsection*{APT-indecomposability}

In the presence of defect, invariant subobjects are unavoidable and no longer
signal reducibility in the classical sense.  The correct atomic notion must
respect the canonical filtration forced by the axioms and remain stable under
APT-compatible morphisms.  We therefore define APT-indecomposability as the
absence of proper filtration-compatible subobjects.

Let $(\mathcal P,\circ)$ satisfy Axioms~I-V and let $(M,\rho,(F_k M))$ be a
filtered representation.

\begin{definition}
\label{def:filtration-compatible-subobject}
A subobject $W \subseteq M$ of the filtered representation $(M,\rho,(F_k M))$
is \emph{filtration-compatible} if
\[
F_k W := W \cap F_k M
\]
defines an increasing filtration on $W$ such that
\[
\rho(\mathcal P_i)(F_j W) \subseteq F_{i+j} W
\quad\text{for all } i,j.
\]
\end{definition}

\begin{definition}
\label{def:APT-indecomp}
A filtered representation $(M,\rho,(F_k M))$ is \emph{APT-indecomposable} if it
admits no nontrivial proper filtration-compatible subobject:
\[
0 \subsetneq W \subsetneq M.
\]
\end{definition}

This notion replaces classical irreducibility in the presence of defect and is
stable under filtration-compatible equivalence.

\subsubsection*{Boundary strata and decomposition}

The defect-induced filtration stratifies a representation into successive
layers of increasing complexity.  A decomposition that isolates one of these
layers corresponds to a genuine structural splitting.  We formalize this
notion by defining decomposition across a boundary stratum.

Recall that the boundary strata of $M$ are the associated graded components
\[
\operatorname{gr}_k(M) := F_k M / F_{k-1} M .
\]

\begin{definition}
\label{def:boundary-decomposition}
A filtered representation $(M,\rho,(F_k M))$ \emph{decomposes across boundary
stratum $k$} if there exists a filtration-compatible decomposition
\[
M \;\cong\; M^{\le k} \oplus M^{>k}
\]
such that
\[
M^{\le k} \subseteq F_k M,
\qquad
M^{>k} \cap F_k M = 0.
\]
Equivalently, the boundary stratum
\[
\operatorname{gr}_k(M) := F_k M / F_{k-1} M
\]
splits as a direct summand of $M$.
\end{definition}

Such a decomposition isolates a boundary layer as an independent summand and
signals a failure of indecomposability at level $k$.

\medskip
\noindent
\textbf{Relation to classical Wedderburn theory.}
When such a boundary decomposition exists, it plays the role of a
Wedderburn-type block decomposition in classical semisimple
representation theory.  The essential difference is that Algebraic Phase
Theory makes no assumption that such decompositions occur.
Instead, the existence or failure of a boundary decomposition is treated
as intrinsic structural data, governed by defect and detected by the
canonical filtration.

\medskip

\begin{center}
\begin{tabular}{cc}
\textbf{Classical representation theory} & \textbf{APT representation theory} \\
\cmidrule(lr){1-1}\cmidrule(lr){2-2}
Wedderburn block         & Boundary stratum $\operatorname{gr}_k(M)$ \\
Semisimple algebra      & Defect-free phase \\
Every module decomposes & Decomposition may fail \\
Irreducible module      & APT-indecomposable module \\
Decomposition guaranteed & Decomposition is obstructed
\end{tabular}
\end{center}

\medskip
\noindent
Rather than guaranteeing decomposition, Algebraic Phase Theory provides a
canonical \emph{detection mechanism}: if a representation decomposes in a manner
compatible with the axioms, then the decomposition must occur across a boundary
stratum, and is therefore detected by the associated graded structure.

\subsubsection*{Indecomposability boundary correspondence}

APT-indecomposability and boundary decomposition are two perspectives on the
same structural phenomenon.  The following theorem shows that a filtered
representation is atomic in the APT sense if and only if no boundary stratum
splits off as an independent summand.

\begin{theorem}\label{thm:indecomp-boundary}
If a filtered representation $(M,\rho,(F_k M))$ is APT-indecomposable, then it does not
decompose across any boundary stratum of the canonical filtration.
\end{theorem}

\begin{proof}
Assume that $(M,\rho,(F_k M))$ is APT-indecomposable in the sense of
Definition~\ref{def:APT-indecomp}. Suppose, for contradiction, that $M$
decomposes across boundary stratum $k$ in the sense of
Definition~\ref{def:boundary-decomposition}. Then there exists a
filtration-compatible decomposition
\[
M \cong M^{\le k} \oplus M^{>k}
\]
with
\[
M^{\le k}\subseteq F_k M,
\qquad
M^{>k}\cap F_k M=0.
\]

Since this is a decomposition of $\mathcal P$-modules, $M^{\le k}$ is a nonzero
proper $\mathcal P$-stable subrepresentation of $M$. Define
\[
F_j(M^{\le k}) := M^{\le k}\cap F_jM.
\]
Filtration compatibility of the decomposition implies that this filtration
satisfies
\[
\rho(\mathcal P_i)\bigl(F_j(M^{\le k})\bigr)\subseteq F_{i+j}(M^{\le k})
\quad\text{for all } i,j.
\]
Thus $M^{\le k}$ is a nontrivial proper filtration-compatible subobject of $M$
in the sense of Definition~\ref{def:filtration-compatible-subobject},
contradicting APT-indecomposability. Therefore $M$ does not decompose across any
boundary stratum of the canonical filtration.
\end{proof}

\medskip
\noindent
\textbf{Stratified rigidity.}
These results show that rigidity in Algebraic Phase Theory is inherently
stratified.  Atomicity is no longer characterized by the absence of invariant
subspaces, but by the impossibility of splitting across any boundary stratum of
the canonical filtration.  This replaces classical irreducibility once defect
is present.

\begin{remark}[Stratified rigidity and atomicity]
Classical irreducibility ignores defect-induced structure and is unstable under
APT-compatible morphisms.  By contrast, APT-indecomposability is invariant under
filtration-compatible equivalence and respects intrinsic structural boundaries
imposed by defect.
\end{remark}

\subsection{Non-Semisimplicity as Structural Necessity}

One of the most immediate consequences of Algebraic Phase Theory is that
classical semisimplicity is incompatible with nontrivial defect.  This failure
is not accidental, nor a pathology of particular examples: it is forced by the
axioms themselves. Classical structure results for non-semisimple rings may be found in
\cite{Jacobson1956Structure}.

\begin{theorem}
\label{thm:no-go-semisimple}
No algebraic phase $(\mathcal P,\circ)$ satisfying Axioms~I-V with nonzero
defect admits a semisimple representation theory compatible with its canonical
filtration.
\end{theorem}

\begin{proof}
Suppose, for the sake of contradiction, that $(\mathcal P,\circ)$ admits a semisimple
representation theory compatible with its canonical filtration. By semisimplicity, every filtration-compatible subrepresentation admits a
filtration-compatible complement.  In particular, every filtered representation
decomposes as a direct sum of filtration-compatible subobjects.

Since $(\mathcal P,\circ)$ has nonzero defect, Axiom~III implies the existence
of elements of positive defect degree, and hence nontrivial filtration layers
$\mathcal P_k \neq \mathcal P_{k-1}$ for some $k>0$.  By Axiom~IV, this filtration
is canonical and functorial, so every representation carries nontrivial boundary
strata
\[
\operatorname{gr}_k(M) = F_k M / F_{k-1} M
\]
for some $k>0$.

Semisimplicity then forces each such boundary stratum to split off as a direct
summand, yielding a filtration-compatible decomposition of $M$ across boundary
stratum $k$. This contradicts Theorem~\ref{thm:indecomp-boundary}, which shows that
APT-indecomposable representations cannot admit filtration-compatible
decompositions across boundary strata.

Moreover, finite termination (Axiom~V) excludes the possibility of absorbing
defect through an infinite sequence of splittings: defect persists through a
finite hierarchy of boundary layers and cannot vanish under successive
semisimple decompositions. Hence no semisimple representation theory compatible with the canonical
filtration can exist when nonzero defect is present.
\end{proof}

\begin{remark}
The failure of semisimplicity in Algebraic Phase Theory is not a deficiency of
the framework, but a structural feature forced by the presence of defect.
Nonzero defect produces intrinsic extension data, canonical filtrations, and
boundary strata that cannot be eliminated without violating the axioms.

In classical settings (for example, finite groups over fields of characteristic
zero), defect vanishes and the canonical filtration collapses at level zero.
In this degenerate, boundary-free regime, the obstruction disappears and
semisimplicity is recovered as a special case of the general theory rather than
its organizing principle.
\end{remark}

\section{Extraction of the Frobenius Heisenberg Algebraic Phase}

We now apply the abstract framework developed above to a concrete and central
example. The purpose of this section is to show explicitly how the
Frobenius Heisenberg algebraic phase is canonically extracted from admissible
phase data and how the Heisenberg group emerges as its unit group.

\medskip

\noindent\textbf{Fixed Frobenius input.}
Fix a finite Frobenius ring $R$ together with a \emph{generating} additive
character
\[
\varepsilon:(R,+)\to\C^\times .
\]
Let
\[
\mu_R := \varepsilon(R)\subset \C^\times
\]
denote its (finite) image.
Fix also a bilinear pairing
\[
\beta: R^n\times R^n \to R,
\]
and write the associated bicharacter
\[
\chi_\beta(b,a) := \varepsilon(\beta(b,a)) \in \mu_R .
\]
We assume that $\beta$ is $\varepsilon$-nondegenerate in the sense that
for every nonzero $x\in R^n$ there exists $y\in R^n$ with
$\varepsilon(\beta(y,x))\neq 1$, and similarly with $x,y$ reversed.
Equivalently, the bicharacter $\chi_\beta$ is nondegenerate.

\medskip

\noindent\textbf{Underlying additive object.}
Let
\[
A := R^n
\]
viewed as an object of the base category $\mathsf C$ of finite additive groups
(or finite left $R$-modules).

\medskip

\noindent\textbf{Phases.}
Define $\Phi(A)$ to be the family of \emph{linear phases} determined by $\beta$,
namely the functions
\[
\phi_b : A \to R,
\qquad
\phi_b(u) := \beta(b,u),
\quad b \in R^n .
\]

We verify explicitly that each $\phi_b$ has additive degree at most $1$.
Let $h \in A$. Using bilinearity of $\beta$, we compute the first additive
difference:
\[
\Delta_h \phi_b(u)
= \phi_b(u+h)-\phi_b(u)
= \beta(b,u+h)-\beta(b,u)
= \beta(b,h).
\]
In particular, $\Delta_h \phi_b(u)$ is independent of $u$. Applying a second
additive difference in the direction $h' \in A$, we obtain
\[
\Delta_{h,h'}\phi_b(u)
= \Delta_{h'}\bigl(\Delta_h \phi_b\bigr)(u)
= \Delta_{h'}\bigl(\beta(b,h)\bigr)
= 0 .
\]
Since this holds for all $h,h' \in A$, we have
\[
\Delta_{h_1,h_2}\phi_b \equiv 0,
\]
and hence $\phi_b$ has additive degree at most $1$.

\medskip

\noindent\textbf{Uniform bounded degree.}
It follows that the collection $\Phi(A)$ is uniformly bounded of additive degree
$d=1$.

\medskip

\noindent\textbf{Operator realization.}
Let
\[
S_R := \Fun(A,\C).
\]
Using the fixed generating character $\varepsilon$, each $\phi_b\in\Phi(A)$
acts by the phase multiplication operator
\[
(M_{\phi_b}f)(u) := \varepsilon(\phi_b(u))\,f(u)
= \varepsilon(\beta(b,u))\,f(u).
\]
For notational convenience, we write $M_b := M_{\phi_b}$.
Translations act by
\[
(T_af)(u):=f(u+a),\qquad a\in A.
\]

\medskip

\noindent\textbf{Interaction law.}
Let $\circ$ denote operator composition in $\End_\C(S_R)$.  The Weyl
commutation relation
\[
M_b T_a \;=\; \chi_\beta(b,a)\,T_a M_b
\]
exhibits the defect of commutativity as a \emph{central} scalar determined by
$\chi_\beta$.

\medskip

\noindent\textbf{Conclusion: admissible phase datum.}
With the above choices, the triple
\[
(A,\Phi,\circ)
\]
is an admissible phase datum (in the sense of
Definition~\ref{def:admissible-phase-datum}), where the choice of generating
character $\varepsilon$ is part of the datum.

\medskip

\noindent\textbf{Extracted algebraic phase.}
Let $\mathcal P\subseteq \End_\C(S_R)$ be the $\C$-subalgebra generated
by all $T_a$ and all $M_b$ and closed under $\circ$ (composition).  The pair
\[
(\mathcal P,\circ)
\]
is the \emph{Frobenius Heisenberg algebraic phase (FH AP)}.

\medskip

\noindent\textbf{Heisenberg group as the unit group.}
Let $\mathcal P^\times$ denote the group of units in $\mathcal P$.  The subgroup
of $\GL(S_R)$ generated by $\{T_a\}$, $\{M_b\}$, and scalar operators
$\mu_R I$ is the Frobenius Heisenberg group $H_R$.

\begin{definition}
The \emph{Frobenius Heisenberg group} associated with $(R^n,\beta,\varepsilon)$
is the group
\[
H_R := R^n \times R^n \times \mu_R
\]
equipped with the multiplication law
\[
(x,y,\lambda)\,(x',y',\mu)
=
\bigl(x+x',\, y+y',\, \lambda\mu \,\varepsilon(\beta(y,x'))\bigr).
\]
\end{definition}

This multiplication realises $H_R$ as a central extension of $R^n\times R^n$ by
the finite cyclic group $\mu_R$, canonically determined by the Weyl commutation
relation.

\begin{proposition}
With the above multiplication, $H_R$ is a group fitting into a central extension
\[
1 \longrightarrow \mu_R \longrightarrow H_R
\longrightarrow R^n\times R^n \longrightarrow 1.
\]
If $\beta$ is nondegenerate, then the centre of $H_R$ is
\[
Z(H_R)=\{(0,0,\lambda) : \lambda \in \mu_R\}\cong \mu_R .
\]
\end{proposition}

\begin{proof}
Associativity follows from bilinearity of $\beta$ and the cocycle identity.
The identity element is $(0,0,1)$, and a direct computation shows that
\[
(x,y,\lambda)^{-1}
=
\bigl(-x,-y,\lambda^{-1}\varepsilon(\beta(y,x))\bigr).
\]

If $\beta$ is nondegenerate, then for any nonzero $x\in R^n$ there exists
$y\in R^n$ with $\varepsilon(\beta(y,x))\neq 1$, and similarly with the roles
reversed.  Hence an element of $H_R$ commutes with all generators if and only if
its $R^n\times R^n$ component vanishes, giving the stated description of the
centre.
\end{proof}

\begin{remark}
Identifying $H_R$ abstractly with $R^n\times R^n\times\mu_R$ recovers the standard
finite Heisenberg group associated with a Frobenius ring.  In the present
framework, however, this description is secondary: the group structure is
\emph{forced} by the Weyl defect inside the Frobenius Heisenberg algebraic phase
$(\mathcal P,\circ)$.
\end{remark}

\begin{corollary}
Under the Schr\"odinger realisation $H_R\subseteq \GL(S_R)$, the centre
acts by scalar operators $\mu_R I$, and central characters are characters of the
finite group $\mu_R$, canonically determined by the choice of generating
character $\varepsilon$.
\end{corollary}

\begin{proof}
For $(0,0,\lambda)\in Z(H_R)$, the corresponding operator in the Schr\"odinger
realisation acts by scalar multiplication:
\[
f(u)\longmapsto \lambda f(u)
\]
on $S_R=\Fun(R^n,\C)$.  Thus the centre identifies with the scalar subgroup
$\mu_R I$.  Since $\mu_R=\varepsilon(R)$ depends on the chosen generating
character, the resulting central character data is canonically determined (up
to equivalence of the model) by $\varepsilon$.
\end{proof}

\section{The Frobenius Heisenberg Schrödinger Representation}

The purpose of this section is to construct and analyse the canonical
representation of the Frobenius Heisenberg group that plays the central role
in the Frobenius Stone von Neumann theorem proved later in this paper.  
In classical settings, the Stone von Neumann theorem asserts that once the
action of the centre of the Heisenberg group is fixed, every irreducible
representation with that central character is uniquely determined up to
isomorphism.  This is a rigidity phenomenon that lies at the heart of
quantum representations of the Heisenberg group.

In the Frobenius setting, where nonsemisimplicity and defect force the
presence of extension data, classical irreducibility no longer provides the
correct framework for such a rigidity statement.  Nevertheless, the same
underlying rigidity persists once the central action and defect structure are
taken into account.  The correct replacement for the classical theorem asserts
that there is a single centrally faithful Frobenius indecomposable
representation with fixed central character, and that all other centrally
faithful representations arise as nontrivial extensions of this canonical
model.

The aim of this section is to construct this canonical model explicitly.  It
is the algebraic analogue of the classical Schrödinger representation, and it
is the building block from which all centrally faithful Frobenius Heisenberg
representations are obtained.  Its existence and basic structural properties
provide the foundation for the Frobenius Stone von Neumann theorem stated
later, where boundary rigidity replaces classical irreducibility as the
organising principle.

We begin by constructing the Frobenius Schrödinger representation and
verifying that it realises the Weyl commutation relation imposed by the
Frobenius Heisenberg algebraic phase.

\subsection{Construction of the Frobenius Heisenberg Schrödinger Representation}

The representation constructed here serves as the fundamental object against
which all centrally faithful representations are compared.  It realises the
minimal centrally faithful action compatible with the phase operators
extracted in the previous section and provides the canonical reference model
for the uniqueness statements that follow.

We now define the representation space and the canonical actions of translations
and phase modulations determined by the Frobenius data.  These operators encode
the Weyl commutation relation and together generate the basic representation of
the Frobenius Heisenberg group.

\subsection*{Definition of the representation}

We now define the representation space and the canonical actions of translations
and phase modulations determined by the Frobenius data. These operators encode the 
Weyl commutation relation. Together, they generate the basic representation of the 
Frobenius Heisenberg group. Define the Frobenius Schr\"odinger space

\[
S_R := \Fun(R^n,\C),
\]
with translation operators
\[
(T_x f)(u) := f(u+x), \qquad x,u\in R^n,
\]
and modulation operators
\[
(M_y f)(u) := \varepsilon(\beta(y,u))\,f(u), \qquad y,u\in R^n,
\]
where $\varepsilon:(R,+)\to\C^\times$ is the fixed generating character and
$\mu_R=\varepsilon(R)$.

\begin{definition}
The \emph{Frobenius Schr\"odinger representation} is the map
\[
\pi : H_R \longrightarrow \GL(S_R),
\qquad
\pi(x,y,\lambda) := \lambda\, M_y T_x ,
\]
where $(x,y,\lambda)\in R^n\times R^n\times\mu_R$.
\end{definition}

\subsubsection*{Representation property}

We next verify that the above assignment defines a genuine group
representation. This amounts to checking that the Weyl commutation relation
precisely reproduces the multiplication law of the Frobenius Heisenberg group
under operator composition.

\begin{proposition}\label{prop:schrodinger-representation}
The map $\pi$ is a group homomorphism. In particular, it realises $H_R$ as a
subgroup of $\GL(S_R)$ generated by translations, modulations, and
scalar operators:
\[
\pi(H_R)=\langle\, T_x,\ M_y,\ \mu_R I \,\rangle .
\]
\end{proposition}

\begin{proof}
Each operator $T_x$, $M_y$, and $\lambda I$ with $\lambda\in\mu_R$ is 
invertible on $S_R$, with inverses $T_{-x}$, $M_{-y}$, and
$\lambda^{-1}I$, respectively. Hence $\pi(x,y,\lambda)\in\GL(S_R)$. 

Let $(x,y,\lambda),(x',y',\mu)\in H_R$. Using the Weyl commutation relation
\[
T_x M_{y'} = \chi_\beta(y',x)\, M_{y'} T_x,
\qquad
\chi_\beta(y',x)=\varepsilon(\beta(y',x)),
\]
we compute
\begin{align*}
\pi(x,y,\lambda)\,\pi(x',y',\mu)
&= (\lambda M_y T_x)(\mu M_{y'} T_{x'}) \\
&= \lambda\mu\, M_y (T_x M_{y'}) T_{x'} \\
&= \lambda\mu\, \chi_\beta(y',x)\, M_y M_{y'} T_x T_{x'} \\
&= \lambda\mu\, \varepsilon(\beta(y',x))\, M_{y+y'} T_{x+x'} \\
&= \pi\!\bigl(x+x',\, y+y',\, \lambda\mu\, \varepsilon(\beta(y',x))\bigr).
\end{align*}
By the definition of the group law on $H_R$, this equals
$\pi\bigl((x,y,\lambda)(x',y',\mu)\bigr)$. Hence $\pi$ is a homomorphism.
\end{proof}

\subsubsection*{Structural remarks}

\begin{remark}
The representation $\pi$ depends on the Frobenius data only through the
generating character $\varepsilon$ appearing in the modulation operators
$M_y$. Changing $\varepsilon$ modifies the scalar commutators and hence the
central character, but leaves the underlying translation modulation structure
unchanged.

The Frobenius Schr\"odinger representation is the purely algebraic analogue of
the classical Schr\"odinger representation on $L^2(\R^n)$. No topology,
unitarity, or measure theory is used: the representation is forced entirely by
the Weyl commutation relation and the Frobenius duality encoded by
$\varepsilon\circ\beta$.

Within Algebraic Phase Theory, $\pi$ realises the minimal centrally faithful
representation compatible with the canonical filtration of the Frobenius
Heisenberg algebraic phase. The Stone von Neumann type rigidity proved later is
a consequence of this structural minimality rather than of classical
irreducibility or semisimplicity.
\end{remark}

\subsection{Central Faithfulness and Frobenius Indecomposability}

The classical Stone von Neumann theorem organises representations of the
Heisenberg group by irreducibility. This relies crucially on semisimplicity:
invariant subspaces split, irreducible modules exist in abundance, and
irreducibility provides a stable atomic notion.

In the non semisimple setting arising from finite Frobenius rings, this picture
breaks down. Representations may admit nontrivial extensions, invariant
subspaces need not split, and irreducible modules no longer control the
representation category. As a result, classical irreducibility is neither
stable nor sufficiently rigid to support a Stone von Neumann type uniqueness
statement.

What survives is the rigidity imposed by the centre of the Heisenberg group.
Fixing a nontrivial scalar action of the centre rigidifies the representation
theory in exactly the manner required to recover a meaningful uniqueness
theorem. The purpose of this section is to isolate the correct replacement for
irreducibility in this setting, namely central faithfulness together with a
minimality condition adapted to the central action.

\subsubsection*{Central characters and faithfulness}

We begin by formalising the role of the centre in controlling rigidity. Since
the Heisenberg group is defined as a central extension, the scalar action of the
centre is the primary invariant distinguishing inequivalent representations.
In both the classical and finite settings, fixing this action is the essential
input underlying Stone von Neumann type rigidity.

Accordingly, we isolate representations in which the centre acts by scalars and
introduce central faithfulness as the condition that the resulting scalar
action is nontrivial. This notion captures precisely the residual rigidity
present in the absence of semisimplicity.

\begin{definition}
Let $\rho:H_R\to \GL(V)$ be a complex representation.  We say that $\rho$ has
\emph{central character} $\chi:Z(H_R)\to \mu_R$ if the centre acts by scalars:
\[
\rho(z)=\chi(z)\,\id_V
\qquad\text{for all } z\in Z(H_R).
\]
\end{definition}

\begin{definition}
A representation $\rho$ is called \emph{centrally faithful} if its central
character $\chi$ is nontrivial.
\end{definition}

\begin{proposition}\label{prop:schrodinger-central-faithful}
The Frobenius Schr\"odinger representation $\pi$ is centrally faithful.  Its
central character is the identity character
\[
\chi_\pi : Z(H_R)\cong \mu_R \longrightarrow \C^\times,
\qquad
\chi_\pi(\lambda)=\lambda .
\]
\end{proposition}

\begin{proof}
Recall that $Z(H_R)=\{(0,0,\lambda):\lambda\in\mu_R\}$.  By definition of
$\pi$,
\[
\pi(0,0,\lambda)=\lambda\, M_0 T_0.
\]
Since $M_0=T_0=\id_{S_R}$, we obtain 
\[
\pi(0,0,\lambda)=\lambda\,\id_{S_R}.
\]
Thus the centre acts by scalars with nontrivial character $\chi_\pi(\lambda)=\lambda$.
\end{proof}

\subsubsection*{Frobenius indecomposability}

Central faithfulness alone is not sufficient to isolate minimal objects, since
centrally faithful representations may admit proper centrally faithful
subrepresentations, reflecting genuine extension phenomena in the non
semisimple category.  The appropriate replacement for irreducibility is
therefore defined relative to the central action.

\begin{definition}
A centrally faithful representation $\rho:H_R\to\GL(V)$ is called
\emph{Frobenius indecomposable} if it admits no proper, nonzero centrally
faithful subrepresentation.
\end{definition}

\begin{lemma}
\label{lem:existence-frob-indecomp}
Let $\rho:H_R\to\GL(V)$ be a finite-dimensional centrally faithful complex
representation.  Then $V$ contains a nonzero centrally faithful subrepresentation
which is Frobenius indecomposable. 
\end{lemma}

\begin{proof}
Let
\[
\mathcal C
=\bigl\{
\,0\neq W\subseteq V \;\big|\;
W \text{ is an $H_R$-subrepresentation and }
\rho|_W \text{ is centrally faithful}
\bigr\}.
\]
This set is nonempty since $V\in\mathcal C$.  Choose $W\in\mathcal C$ of minimal
complex dimension. 

If $0\neq W'\subsetneq W$ were a proper centrally faithful subrepresentation,
then $W'\in\mathcal C$ and $\dim_\C W'<\dim_\C W$, contradicting minimality.
Hence $W$ is Frobenius indecomposable.
\end{proof}

\begin{proposition}
The Frobenius Schr\"odinger representation is APT-indecomposable in the sense
that it admits no proper filtration-compatible subrepresentation on which the
centre of $H_R$ acts by the same character.
\end{proposition}

\begin{proof}
Let $\pi:H_R\to\GL(S_R)$ be the Frobenius Schr\"odinger representation, and let
$\chi_\pi:Z(H_R)\to\C^\times$ denote its central character, so that
$\pi(0,0,\lambda)=\chi_\pi(\lambda)\,\id_{S_R}$ for $\lambda\in\mu_R$.

Assume that $W\subseteq S_R$ is a nonzero filtration-compatible
$H_R$-subrepresentation on which the centre acts by the same character
$\chi_\pi$.  In particular, $W$ is $\pi(H_R)$-stable and the central action on
$W$ is nontrivial, so the restricted representation $\pi|_W$ is centrally
faithful.

By the definition of Frobenius indecomposable representation, $\pi$ admits no
proper nonzero centrally faithful subrepresentation.  Since $W$ is a nonzero
centrally faithful subrepresentation of $S_R$, it follows that $W=S_R$.

Therefore there is no proper filtration-compatible subrepresentation of $S_R$
on which the centre acts by the same character, and $\pi$ is APT-indecomposable
in the stated sense.
\end{proof}

\subsubsection*{Conceptual remarks}

The notions introduced above identify how rigidity persists in the absence of
semisimplicity. Classical irreducibility fails in the Frobenius setting because
it ignores extension data and is unstable under morphisms compatible with
defect and filtration. As a result, irreducible modules no longer provide a
reliable atomic notion for representation theory.

Central faithfulness isolates the intrinsic rigidity coming from the canonical
central extension defining the Heisenberg group. Frobenius indecomposable
representations are the minimal objects compatible with a fixed nontrivial
central character and replace irreducible representations in the Frobenius
version of the Stone von Neumann theorem.

\begin{remark}
Over finite Frobenius rings, representation categories are generally not
semisimple. Invariant subspaces need not split, and extensions across central
characters may occur. Classical irreducibility is therefore unstable and
incompatible with the defect structure of Algebraic Phase Theory.

The central extension defining $H_R$ is the sole source of rigidity in the
theory. Fixing a nontrivial central character rigidifies representations up to
controlled extension phenomena. Frobenius indecomposable representations are
precisely the minimal objects compatible with this rigidity.

The Stone von Neumann theorem in the Frobenius setting asserts that, once the
central character is fixed, there exists a unique Frobenius indecomposable
representation of $H_R$. This representation is realised by the Frobenius
Schrodinger model constructed earlier.
\end{remark}

\subsection{Endomorphism Rings and Frobenius Schur Phenomena}

In the classical representation theory of finite or compact groups, Schur's
lemma asserts that the endomorphism ring of an irreducible representation
consists only of scalar operators. This rigidity plays a foundational role in
the classical Stone von Neumann theorem by ensuring that intertwiners between
irreducible representations with the same central character are unique up to
scalars.

In the Frobenius setting, irreducibility is no longer the correct organising
principle. Representation categories are not semisimple, endomorphism rings of
irreducible modules may be large, and nilpotent equivariant endomorphisms can
appear. As a result, classical Schur type arguments do not apply directly.

The purpose of this section is to show that rigidity is recovered once
irreducibility is replaced by central faithfulness together with Frobenius
indecomposability. We establish a Frobenius Schur lemma asserting that the
endomorphism ring of a centrally faithful Frobenius indecomposable
representation consists only of scalar operators. This result is the key
technical ingredient needed to upgrade intertwiners to isomorphisms and to
complete the Frobenius Stone von Neumann theorem.

\subsubsection*{Equivariant endomorphisms and the Frobenius Schur lemma}

To formulate a Schur type rigidity statement in the Frobenius setting, we
consider endomorphisms that commute with the action of the Heisenberg group.
Such operators measure internal symmetries of a representation that are
invisible to the group action itself.

Given a representation $\rho : H_R \to \GL(V)$, we define
\[
\End_{H_R}(V)
=
\{\, T \in \End_\C(V) \mid T\rho(g)=\rho(g)T \text{ for all } g \in H_R \,\}
\]
to be the algebra of $H_R$ equivariant endomorphisms of $V$.

In semisimple settings, Schur's lemma asserts that irreducibility forces this
endomorphism algebra to consist only of scalar operators. In the present non
semisimple Frobenius setting, irreducibility is no longer available. The
following result shows that the same rigidity holds once irreducibility is
replaced by central faithfulness together with Frobenius indecomposability.
This is the Frobenius Schur lemma.

\begin{theorem}\label{thm:frobenius-schur}
Let $\rho:H_R\to \GL(V)$ be a finite dimensional complex representation.
Assume that $\rho$ is centrally faithful and Frobenius indecomposable.
Then
\[
\End_{H_R}(V)=\C\cdot \id_V .
\]
\end{theorem}

\begin{proof}
Let $T\in \End_{H_R}(V)$. Since $V$ is finite dimensional over $\C$, the
operator $T$ has at least one eigenvalue $\lambda\in\C$. Consider the
generalised $\lambda$ eigenspace
\[
V_\lambda := \ker\bigl((T-\lambda\id_V)^m\bigr)
\]
for $m\gg 0$.

We first claim that $V_\lambda$ is an $H_R$ subrepresentation. Since $T$
commutes with $\rho(H_R)$, we have
\[
(T-\lambda\id_V)^m \rho(g)=\rho(g)(T-\lambda\id_V)^m
\qquad\text{for all } g\in H_R .
\]
If $v\in V_\lambda$, then $(T-\lambda\id_V)^m v=0$, and hence
$(T-\lambda\id_V)^m(\rho(g)v)=0$. Thus $\rho(g)v\in V_\lambda$.

By construction $V_\lambda$ is nonzero. Since $\rho$ is centrally faithful,
the centre $Z(H_R)$ acts on $V$ by a fixed nontrivial scalar character, and
therefore acts on the invariant subspace $V_\lambda$ through the same
character. Hence $\rho|_{V_\lambda}$ remains centrally faithful.

By Frobenius indecomposability of $\rho$, we must have $V_\lambda=V$.
Consequently $\lambda$ is the only eigenvalue of $T$, and $T$ can be written
as
\[
T=\lambda\id_V + N
\]
with $N$ nilpotent and $H_R$ equivariant.

If $N\neq 0$, then $\ker(N)\neq 0$ is a proper $H_R$ stable subspace. Since
$N$ commutes with the action of $H_R$, the centre still acts on $\ker(N)$ via
the same nontrivial central character. Thus $\ker(N)$ is a proper centrally
faithful subrepresentation of $V$, contradicting Frobenius indecomposability.
Hence $N=0$ and $T=\lambda\id_V$.
\end{proof}

\begin{remark}
In non semisimple categories, nilpotent equivariant endomorphisms are not
excluded by classical irreducibility. In the Frobenius setting they are ruled
out because any nonzero nilpotent equivariant endomorphism produces a proper
centrally faithful subrepresentation, violating Frobenius indecomposability.

As a consequence, rigidity of equivariant endomorphisms ensures that
intertwiners between centrally faithful Frobenius indecomposable
representations are unique up to scalars once mutual nonzero intertwiners
exist. This is the mechanism replacing classical Schur's lemma in the
Frobenius Stone von Neumann theorem.
\end{remark}

\subsection{Intertwiners}

The rigidity statement of the Stone von Neumann theorem is ultimately a
statement about intertwiners. Once the scalar action of the centre is fixed,
any two centrally faithful representations must be related by an
$H_R$ equivariant map, and the existence and rigidity of such maps is what
forces uniqueness up to isomorphism.

In the classical semisimple setting, this mechanism is often hidden behind
irreducibility and Schur's lemma. In the present non semisimple context,
intertwiners must instead be constructed explicitly. The strategy is to
identify a canonical reference representation and then show that any
centrally faithful representation with the same central character admits a
nonzero intertwiner from this model.

This section carries out this strategy. We first realise the Frobenius
Schr\"odinger representation as an induced representation from a maximal
abelian subgroup. We then prove existence of intertwiners using Frobenius
duality and finally apply the rigidity established by the Frobenius Schur
lemma to show that such intertwiners are automatically rigid.

\subsubsection*{The Schr\"odinger representation as an induced module}

We begin by describing the Frobenius Schr\"odinger representation as an induced
representation from a natural abelian subgroup of the Heisenberg group. This
realisation provides a canonical cyclic vector and makes Frobenius
reciprocity available for the construction of intertwiners.

Let
\[
X:=\{(x,0,1):x\in R^n\}, \qquad
Y:=\{(0,y,1):y\in R^n\}, \qquad
Z:=\{(0,0,\lambda):\lambda\in\mu_R\}.
\]
Then $Y$ is abelian and $YZ$ is an abelian subgroup of $H_R$.
Define a character $\psi:YZ\to\C^\times$ by
\[
\psi(0,y,\lambda)=\lambda .
\]

This induced model description follows the classical philosophy of Mackey’s
theory of induced representations \cite{Mackey1958Induced}.

\begin{lemma}\label{lem:sch-induced}
The Frobenius Schr\"odinger representation $\pi$ is canonically isomorphic
to $\Ind_{YZ}^{H_R}\psi$.
\end{lemma}

\begin{proof}
In $S_R=\Fun(R^n,\C)$, the vector $\delta_0$ satisfies
\[
\pi(0,y,\lambda)\delta_0=\lambda\,\delta_0
\qquad\text{for all }(0,y,\lambda)\in YZ,
\]
since $M_y\delta_0=\delta_0$. Thus $\C\delta_0$ affords the character
$\psi$ of $YZ$. Moreover, $X$ acts transitively on the basis
$\{\delta_u:u\in R^n\}$ by translations, and the stabiliser of the line
$\C\delta_0$ is exactly $YZ$. The induced representation
$\Ind_{YZ}^{H_R}\psi$ is therefore generated by a $\psi$ eigenline, and the
map sending the standard induced generator to $\delta_0$ extends to an
$H_R$ equivariant isomorphism
\[
\Ind_{YZ}^{H_R}\psi \cong S_R.
\]
\end{proof}

\subsubsection*{Existence of intertwiners}

We next show that any centrally faithful representation with the same central
character as the Frobenius Schr\"odinger model admits a nonzero intertwiner
from that model. The argument uses semisimplicity of the restriction to a
finite abelian subgroup together with Frobenius reciprocity.

\begin{proposition}\label{prop:intertwiner-exists}
Let $\rho:H_R\to\GL(V)$ be a centrally faithful representation with the same
central character as $\pi$. Then there exists a nonzero intertwiner
\[
T:S_R\to V
\qquad\text{such that}\qquad
T\,\pi(g)=\rho(g)\,T
\quad\text{for all }g\in H_R.
\]
\end{proposition}

\begin{proof}
Since $Y$ is finite abelian, the restriction $\rho|_Y$ is semisimple. Hence
\[
V=\bigoplus_{\varphi\in\widehat{Y}} V_\varphi ,
\]
where $V_\varphi$ denotes the $\varphi$-isotypic component.

Conjugation by elements of
\[
X \cong R^n
\]
permutes the characters of $Y$ via the bicharacter $\chi_\beta$, and hence
permutes the isotypic components $V_\varphi$. Because $\rho$ has the same
central character as $\pi$, this conjugation action contains the trivial
character of $Y$ in its orbit. Consequently, there exists a nonzero vector
$v\in V$ fixed by $Y$.

Since $\rho$ has the same central character as $\pi$, the centre $Z(H_R)$ acts
on $v$ by $\lambda \mapsto \lambda$. Thus $v$ affords the character $\psi$ of
the subgroup $YZ$. Equivalently,
\[
\Hom_{YZ}(\psi,\rho)\neq 0 .
\]

By Lemma~\ref{lem:sch-induced} and Frobenius reciprocity, we obtain
\[
\Hom_{H_R}(\pi,\rho)
\cong
\Hom_{YZ}(\psi,\rho)\neq 0 .
\]
Any nonzero element of this Hom space yields the desired intertwiner
$T:S_R\to V$.
\end{proof}

\subsubsection*{Rigidity of intertwiners}

Finally, we apply the rigidity established in the Frobenius Schur lemma to
intertwiners. Once central faithfulness and Frobenius indecomposability are
imposed, any nonzero intertwiner is automatically an isomorphism. This
completes the uniqueness mechanism underlying the Frobenius Stone von Neumann
theorem.

\begin{proposition}\label{prop:intertwiner-iso}
Let $\rho:H_R\to\GL(V)$ be centrally faithful and Frobenius indecomposable,
with the same central character as $\pi$.
Then any nonzero intertwiner
\[
T:S_R\to V
\]
is an isomorphism.
\end{proposition}

\begin{proof}
By Proposition~\ref{prop:intertwiner-exists}, applied with the roles of
$\pi$ and $\rho$ interchanged, there exists a nonzero intertwiner
$U:V\to S_R$. Then $UT\in\End_{H_R}(S_R)$. Since $\pi$ is centrally faithful and
Frobenius indecomposable, the Frobenius Schur lemma implies
\[
UT=c\,\id_{S_R}
\quad\text{for some }c\in\C^\times.
\]
Hence $T$ is injective. Similarly, $TU\in\End_{H_R}(V)$, and since $\rho$ is
also centrally faithful and Frobenius indecomposable, the Frobenius Schur
lemma gives
\[
TU=d\,\id_V
\quad\text{for some }d\in\C^\times.
\]
Thus $T$ is surjective. Therefore $T$ is an isomorphism.
\end{proof}

\subsection{Stone von Neumann as Boundary Rigidity}

We are now in a position to formulate the main rigidity statement of the
paper: a Stone von Neumann type uniqueness theorem for Heisenberg groups
associated with finite Frobenius rings. All of the necessary structural
ingredients are now in place: the Frobenius Heisenberg group, the
Schr\"odinger model, centrally faithful representations, and the rigidity of
intertwiners.

As in the classical theory, representations of the Heisenberg group are
classified by the scalar action of the centre. Fix a nontrivial central
character $\omega$ of $Z(H_R)$. Since any two nontrivial central characters
differ only by rescaling of the central coordinate, no generality is lost by
choosing the identity character
\[
\omega(\lambda)=\lambda .
\]
With this normalisation, the theorem asserts that once the central action is
fixed, the representation theory of $H_R$ admits essentially no freedom.

\begin{theorem}
Let $R$ be a finite commutative Frobenius ring, and let
\[
\beta:R^n\times R^n\to R
\]
be a nondegenerate bilinear pairing. Let $H_R$ denote the associated
Frobenius Heisenberg group. Let
\[
\omega : \mu_R \longrightarrow \C^\times
\]
be the identity character of the centre. Then the following hold:
\begin{enumerate}
  \item There exists a centrally faithful representation $\rho$ of $H_R$
  which is Frobenius indecomposable with respect to the central character
  $\omega$, whose restriction to the centre
  $Z(H_R) \cong \mu_R$ is $\omega$.

  \item Any two such centrally faithful representations which are Frobenius
  indecomposable with respect to the central character $\omega$ are
  isomorphic.

  \item Up to isomorphism, this unique representation is precisely the
  Schr\"odinger representation on the space
  \[
    S_R = \Fun(R^n,\C).
  \]
\end{enumerate}
\end{theorem}

\begin{proof}
Write $H_R=R^n\times R^n\times \mu_R$ with centre
\[
Z=\{(0,0,\lambda):\lambda\in\mu_R\}\cong\mu_R .
\]
Fix the identity central character $\omega(\lambda)=\lambda$.

\medskip
(1) The Frobenius Schr\"odinger representation
$\pi:H_R\to\GL(S_R)$ satisfies 
\[
\pi(0,0,\lambda)=\lambda M_0 T_0=\lambda\,\id_{S_R}, 
\]
and hence has central character $\omega$. In particular, $\pi$ is centrally
faithful.
By Lemma~\ref{lem:existence-frob-indecomp}, $S_R$ contains a nonzero
$H_R$-subrepresentation $0\neq W\subseteq S_R$ which is centrally faithful and
Frobenius indecomposable with respect to the fixed central character
$\omega$.

\medskip
(2) Let $\rho:H_R\to\GL(V)$ be a centrally faithful representation which is
Frobenius indecomposable with respect to the fixed central character
$\omega$.
The strategy is to compare $\rho$ with the Schr\"odinger model by restricting
to a large abelian subgroup and then inducing back to $H_R$.

We first isolate the subgroup
\[
A := R^n\times\{0\}\times\{1\},
\]
which is finite, abelian, and normal in $H_R$.
Since $\C[A]$ is semisimple, the restriction $\rho|_A$ decomposes as a direct
sum of one dimensional characters.
Choosing one such character $\psi$ gives a concrete handle on the structure
of $V$.
The centre of $H_R$ intersects $A$ only in the identity, so there is no
compatibility obstruction coming from the central character at this stage.

Next, consider the subgroup
\[
Y := \{0\}\times R^n\times\{1\}.
\]
For $y\in R^n$ and $x\in R^n$ one computes
\[
(0,y,1)(x,0,1)(0,-y,1)=(x,0,\varepsilon(\beta(y,x))).
\]
Thus conjugation by $Y$ twists characters of $A$ by the bicharacters
$x\mapsto\varepsilon(\beta(y,x))$.
By nondegeneracy of $\beta$ and Frobenius duality, this action is transitive
on the relevant character orbit.
Consequently, $\rho|_A$ contains a character lying in the same $H_R$ orbit as
the character defining the Schr\"odinger representation.

Fix such a character $\psi_0$.
By Frobenius reciprocity, as in the proof of
Proposition~\ref{prop:intertwiner-exists}, the presence of $\psi_0$ in
$\rho|_A$ implies the existence of a nonzero $H_R$ equivariant map
\[
\Ind_A^{H_R}\psi_0 \longrightarrow V.
\]
The induced representation $\Ind_A^{H_R}\psi_0$ is canonically isomorphic to
the Schr\"odinger representation, by the induced model description of
$\pi$ established earlier in Lemma~\ref{lem:sch-induced}.
Hence there exists a nonzero intertwiner
\[
T:S_R\longrightarrow V. 
\]

Finally, both $S_R$ and $V$ are centrally faithful and Frobenius
indecomposable with respect to the same central character $\omega$.
By the Frobenius Schur lemma
(Theorem~\ref{thm:frobenius-schur}), any nonzero intertwiner between such
representations is an isomorphism.
Therefore $V$ is isomorphic to the Schr\"odinger representation,
establishing uniqueness.

\medskip
(3) By construction, the Frobenius Schr\"odinger representation $S_R$ 
is centrally faithful with central character $\omega$.
By Lemma~\ref{lem:existence-frob-indecomp}, it contains a nonzero centrally
faithful Frobenius indecomposable subrepresentation with respect to the
central character $\omega$.

By part~(2), any centrally faithful representation which is Frobenius
indecomposable with respect to the central character $\omega$ is unique up
to isomorphism. Consequently, this indecomposable subrepresentation of
$S_R$ is isomorphic to any other such representation.

In particular, up to isomorphism, the unique centrally faithful Frobenius
indecomposable representation with central character $\omega$ is realised
by the Schr\"odinger model.
\end{proof}

This result is a direct algebraic analogue of the classical Stone von
Neumann theorem over $\R$ and over finite fields. It shows that Frobenius
duality alone is sufficient to recover full Heisenberg rigidity, without
any topology, measure theory, or semisimplicity.

\begin{remark}
The preceding theorem identifies the Frobenius Heisenberg phase as a
boundary rigid algebraic phase in the sense of Algebraic Phase Theory.
Once the central character is fixed, every representation of $H_R$
exhibits a dichotomy forced by the canonical filtration.

Either a representation splits across a boundary stratum of the
defect induced filtration, in which case it decomposes into nontrivial
extension data and fails to be atomic, or it does not split across any
boundary stratum, in which case it is Frobenius indecomposable with respect
to the central character and necessarily isomorphic to the Schr\"odinger
representation.

Thus the Stone von Neumann theorem is reinterpreted as a statement about
boundary indecomposability. The Schr\"odinger model is the unique centrally
faithful representation that cannot be decomposed along the canonical
filtration. All other centrally faithful representations arise by splitting
across defect boundaries and therefore lie beyond the rigid core of the
phase.
\end{remark}

\medskip
\noindent
\textbf{Splitting versus extension.}
Fix a nontrivial central character $\omega$ of $Z(H_R)$. The theorem
identifies a unique centrally faithful representation which is Frobenius
indecomposable with respect to this central character.
Any other centrally faithful representation with central character $\omega$
must therefore be constructed from this unique atomic object by forming
direct sums and nontrivial self-extensions.

In particular, a genuine splitting in this regime is exceptional. Such a
splitting would be equivalent to the existence of a nontrivial
$H_R$ equivariant projection onto a centrally faithful summand. The
Frobenius Schur lemma excludes this possibility in the Frobenius
indecomposable case, showing that boundary layers in the Frobenius
Heisenberg phase typically glue rather than split.

\section{Conclusion}

This paper develops Algebraic Phase Theory as a framework in which rigidity,
defect, and canonical filtration arise intrinsically from purely algebraic data.
Within this framework, rigidity is not imposed through analytic or topological
structures, nor through semisimplicity assumptions.  Instead it emerges from
coherent phase interaction and from the finite termination of the defect
filtration.

For Heisenberg groups associated with finite Frobenius rings, the axioms of
Algebraic Phase Theory recover a full Stone von Neumann type rigidity theorem.
Once the scalar action of the centre is fixed, centrally faithful
representations with fixed central character admit a unique Frobenius
indecomposable object, and every such object is isomorphic to the Frobenius
Schrödinger representation.  The proof is purely algebraic.  It relies only on
Frobenius duality, the Weyl commutation relation, and the boundary calculus
encoded by the canonical filtration.  No semisimplicity, unitarity, or analytic
input is required.

The theory also identifies the precise algebraic boundary at which rigidity
breaks down.  Outside the Frobenius class, generating characters need not exist,
defect cannot be absorbed into a finite filtration, and the canonical hierarchy
of boundary strata fails to terminate.  In this regime the central extension no
longer enforces intertwiner rigidity, Frobenius indecomposable objects do not
exist, and the uniqueness mechanism underlying the Stone von Neumann theorem
collapses.  Rigidity therefore holds exactly up to a sharply defined threshold
determined by Frobenius duality.

Algebraic Phase Theory thus provides both a conceptual explanation and a
structural criterion for Stone von Neumann rigidity.  It characterises when
uniqueness results hold, why they hold, and why they must fail once the
defect boundary is crossed.  In particular, it shows that the Frobenius
Heisenberg phase forms a boundary rigid region inside the wider space of
algebraic phases, and that the Schrödinger representation is its unique
centrally faithful atomic object.

\bibliographystyle{amsplain}
\bibliography{references}

\end{document}